\documentclass[10pt]{amsart}
\usepackage{amsmath, amssymb}
 \usepackage{mathrsfs}
\newcommand{\no}[1]{#1}
\renewcommand{\no}[1]{}
\no{\usepackage{times}\usepackage[subscriptcorrection,
 slantedGreek, nofontinfo]{mtpro}
\renewcommand{\Delta}{\upDelta}}
\usepackage{color}

\date{\today}
\newtheorem{lemma}{Lemma}[section]
\newtheorem{remark}{Remark}[section]

\newtheorem{corollary}{Corollary}[section]
\newtheorem{proposition}{Proposition}[section]
\newtheorem{theorem}{Theorem}[section]

\def\H{\mathcal{H}}

\newcommand{\be}{\begin{equation}}
\newcommand{\ee}{\end{equation}}
\newcommand{\ba}{\begin{array}}
\newcommand{\ea}{\end{array}}

\newcommand{\bea}{\begin{eqnarray*}}
\newcommand{\eea}{\end{eqnarray*}}
\newcommand{\bean}{\begin{eqnarray}}
\newcommand{\eean}{\end{eqnarray}}

\def\hat{\widehat}
\def\tilde{\widetilde}

\def\u{\mathfrak u}

 \def\f{\mathfrak f}
 \def\H{\mathfrak H}

\def\cydot{\leavevmode\raise.4ex\hbox{.}}

\title[]{identification of an inclusion in multifrequency
electric impedance tomography}

\author{Habib Ammari }

\address{Department of Mathematics, ETH Z\"urich, 
R\"amistrasse 101, 8092 Z\"urich, Switzerland}
\email{habib.ammari@math.ethz.ch}

\author{ Faouzi Triki}

\address{Laboratoire Jean Kuntzmann, Universit\'e 
Grenoble-Alpes \& CNRS, 700 Avenue Centrale,
38401 Saint-Martin-d'H\`eres, France}
\email{ faouzi.triki@imag.fr}

\date{}
\begin{document}

\begin{abstract}
The multifrequency  electrical impedance tomography is considered  in order to image a conductivity inclusion inside a homogeneous background medium by injecting one current. An original spectral
decomposition of the solution of the forward conductivity problem is
used to retrieve  the Cauchy data corresponding to the extreme case of perfect conductor. 
Using results based on the unique continuation we then prove the  uniqueness  of multifrequency  electrical impedance tomography and obtain rigorous stability  estimates. Our 
results in this paper are quite surprising in inverse conductivity 
problem since in general  infinitely many   {input
    currents}   
are needed in order to  obtain the  uniqueness in the determination
of the conductivity.  
\end{abstract}

\maketitle

\section{The Mathematical Model and main results}
In this section we introduce the mathematical model of
the multifrequency electrical impedance tomography (mfEIT).
Let $\Omega$ be the open bounded  smooth domain in $\mathbb R^d, d=2, 3$, 
occupied by the sample under investigation and
denote by $\partial \Omega$ its boundary. The mfEIT forward problem 
is to determine the potential $u(\cdot, \omega) \in H^1(\Omega)
:=\{ v \in L^2(\Omega): \nabla v \in L^2(\Omega)\}$,
solution to 
\bean\label{maineq}
\left \{
\ba{lllcc}
-\nabla \cdot \left(\sigma(x, \omega) \nabla u(x, \omega)\right) =  0 & \textrm{in}
\quad \Omega, \\
\sigma(x, \omega) \partial_{\nu_\Omega} u(x, \omega)(x) = f(x) & \textrm{on}
\quad \partial \Omega,\\
\int_{\partial \Omega} u(x, \omega) ds = 0,&
\ea
\right.
\eean
where $\omega$ denotes the frequency, $\nu_\Omega(x)$ is the outward 
normal vector
to $\partial \Omega$,  $\sigma(x, \omega)$ is 
the conductivity distribution,
and   $f\in H^{\frac{1}{2}}_\diamond(\partial \Omega) :=\{ g \in 
 H^{\frac{1}{2}}(\partial \Omega): \int_{\partial \Omega} g \, ds =0\}$  is the input current. \\

In this paper we are interested in the case where  the frequency dependent 
conductivity distribution takes the form 
\bean \label{conductivitydistribution}
\sigma(x, \omega) = k_0+ (k(\omega)-k_0) \chi_D(x)
\eean
with $\chi_D(x)$ being the characteristic function of a smooth  inclusion 
$D$ in $\Omega$ ($\overline D\subset \Omega$), {$k(\omega): \mathbb
R_+\rightarrow
\mathbb C\setminus \overline{\mathbb R_-}$,} being a 
continuous complex-valued function, and $k_0$
being a fixed positive constant (the conductivity of the background medium).  \\

The  mfEIT inverse problem is to recover the shape and  the position
of the inclusion  $D$ from measurements of the boundary voltages 
$u(x, \omega)$ on  $
\partial \Omega $ for
$\omega \in (\underline \omega, \overline \omega)$, 
$0\leq \underline \omega < \overline \omega$. It has many important applications in biomedical imaging. Experimental research has found that the conductivity of many biological tissues 
varies strongly with respect to the
frequency within certain frequency ranges \cite{GPG}. In \cite{AGGJS}, using homogenization techniques, the authors analytically exhibit the fundamental mechanisms underlying the fact that effective biological tissue electrical properties and their
frequency dependence reflect the tissue composition and physiology. There have been also several numerical studies on frequency-difference imaging. It was numerically shown that the approach can accommodate geometrical errors, including imperfectly known boundary \cite{alberti,JS,malone}.\\

 \subsection{Main results}
 Now, we  introduce the class of inclusions on
 which  we will  study the uniqueness and stability of the 
mfEIT inverse problem. Without loss of generality we further
assume that $\Omega$ contains the origin. 

 Let $b_1= \textrm{dist}(0, \partial \Omega)$ and let $b_0< b_1$.
 For $\delta>0$ small enough, and $m>0$ large enough,  define 
 the set of inclusions:
\bea
\mathfrak D := \left\{ D :=\{x\in \mathbb R^d: |x|< 
\Upsilon(\hat x), \hat x = \frac{x}{|x|} \};\;  
b_0<\Upsilon(\hat x)< b_1-\delta;\;
\|\Upsilon\|_{C^{2, \varsigma}} \leq m, \; \varsigma >0 \right\}. 
\eea 

Then,  the 
mfEIT inverse problem has  a unique solution within the
class  $\mathfrak D$, and we have the following stability estimates.

 \begin{theorem} \label{mainthm1}
 Let $D$ and $\widetilde D$ be two  inclusions 
in $\mathcal D$.
Denote by $u$ (resp. $\tilde u$) the solution of 
\eqref{maineq} with 
the inclusion
$D$ (resp. $\widetilde D$). Let 
$$\varepsilon = \sup_{x \in \partial \Omega, \omega \in (\underline \omega,
\overline \omega)}|u-\tilde u |.$$ 

Then,  there exist constants $C>0$ and $\tau\in(0,1)$,  such that  the following
estimate holds:
\bean \label{stabilityestimate1}
\left| D\Delta \widetilde D \right| \leq C 
\left(\frac{1}{\ln(\varepsilon^{-1})}\right)^\tau,
\eean
Here, $\Delta$ denotes the symmetric difference and  the constants $C$ and
$\tau $  depend only on 
$f, \Omega, \mathfrak D,$ 
and $\Sigma:= \{k(\omega);  \omega \in (\underline \omega,
\overline \omega) \}$. 

\end{theorem}

\begin{theorem} \label{mainthm2} Assume that $d=2$, and
let $D$ and $\widetilde D$ be two analytic  inclusions 
in $\mathfrak D$.
Denote by $u$ (resp. $\tilde u$) the solution of \eqref{maineq} with 
the inclusion
$D$ (resp. $\widetilde D$). Let 
$$\varepsilon = \sup_{x \in \partial \Omega, \omega \in (\underline \omega,
\overline \omega)}|u-\tilde u |.$$ 
Then,  there exist constants $C>0$ and $\tau^\prime \in (0,1)$,  such 
that  the following
estimate
\bean \label{stabilityestimate1b}
\left| D\Delta \widetilde D \right| \leq C \varepsilon^{\tau^\prime},
\eean
holds. Here the constants $C$ and $\tau^\prime$ depend only on 
$f, \Omega, \mathfrak D$,  
and $\Sigma$. 

\end{theorem}
  
{
These results  show that the reconstruction 
of the inclusion from multi-frequency boundary voltage data
is improving according to the regularity of the boundary
of the inclusion. Precisely, the stability  estimates 
vary from  logarithmic  to H\"older. They can also be extended to 
a larger class of inclusions as non-star shaped domains, and
to measurements on only a small part of the boundary. In this
paper for the sake of simplicity we do not handle such general 
cases.  } \\

{
The rest of the paper is organized as follows. In section 2, we
introduce the variational  Poincar\'e operator. We study 
in section 3  the regularity of the potential $u(x, \omega)$ 
as a function of the frequency function $k(\omega)$. Precisely, using 
a spectral decomposition based on  the eigenfunctions
of the variational  Poincar\'e operator, we split the
potential $u(x, \omega)$ into a frequency part $u_f(x, k(\omega))$
and a non-frequency part $k_0^{-1} u_0(x)$ 
(Theorem \ref{freqdepend}). Then, we recover the boundary 
Cauchy data for the non-frequency part from the boundary voltage
data (Corollary \ref{estimationu0}). In section 4, we recover the shape and
location of the inclusion from the knowledge of the boundary
Cauchy data of the non-frequency part $k_0^{-1} u_0(x)$, and prove
finally the main results of the paper.  }

\section{The variational 
Poincar\'e operator}
We first introduce an operator whose  spectral decomposition 
will be later  the corner stone  of the identification of the inclusion $D$.
Let $H^1_\diamond(\Omega)$ be the space of functions $v$ in  
$ H^1(\Omega)$ satisfying $\int_{\partial \Omega } v ds = 0$.\\

For $u\in H^1_\diamond(\Omega)$, we infer from the Riesz theorem that 
there exists 
a unique function $Tu\in H^1_\diamond(\Omega)$ such that for 
all $v\in H^1_\diamond(\Omega)$,  
\bean
\int_\Omega \nabla Tu \cdot \nabla v dx = \int_D \nabla u \cdot \nabla v dx. 
\eean
The variational Poincar\'e operator $T: H^1_\diamond(\Omega)\rightarrow
 H^1_\diamond(\Omega)$
is easily seen to be self-adjoint 
and bounded with norm $\|T\| \leq 1$. \\

The spectral problem for $T$  reads
as: 
Find $(\lambda, w) \in \mathbb R\times H^1_\diamond(\Omega)$, $w\not=0$
such that $\forall v \in H^1_\diamond(\Omega)$,
\bea
\lambda \int_\Omega \nabla w \cdot \nabla v dx = \int_D \nabla w \cdot \nabla v dx.
\eea

Integrating by parts, one immediately obtains that any eigenfunction 
$w$ is harmonic in $D$ and in $D^\prime = \Omega\setminus \overline D$, and
satisfies the transmission  and boundary conditions
\bea
w|^+_{\partial D} = w|^-_{\partial D}, \qquad 
\partial_{\nu_D} w|^+_{\partial D}  = 
(1-\frac{1}{\lambda})\partial_{\nu_D} w|^-_{\partial D}, \qquad  
 \partial_{\nu_\Omega} w = 0,
\eea
where $w|^\pm_{\partial D} (x)= 
\lim_{t\rightarrow 0} w(x\pm t\nu_D(x))$ for $x\in\partial D$.   
In other words, $w$ is a solution to \eqref{maineq} for 
$k = k_0(1-\frac{1}{\lambda})$ and $f=0$.\\

Let $\mathfrak H_\diamond$  the space  of harmonic functions in $D$ and
$D^\prime$, with zero mean $\int_{\partial \Omega} u ds(x) =0$, and
zero normal derivative $\partial_{\nu_\Omega} u = 0$ on $\partial \Omega$,
and with finite energy semi-norm
\bea
\|u\|_{\mathfrak H_\diamond} = \int_\Omega |\nabla u|^2 dx.
\eea 
Since the functions in  $\mathfrak H_\diamond$ are harmonic
in $D^\prime$, the  
$\mathfrak H_\diamond$
is a closed subspace of $H^1(\Omega)$. Later on, we will give
a new characterization of the space $\mathfrak H_\diamond$ in 
terms of the single layer potential on  $\partial D$ associated with the 
Neumann function of $\Omega$. 
 \\

 We remark that $Tu = 0$ for all
$u$ in $H^1_0(D^\prime)$, and $Tu = u$ for all
$u$ in $H^1_0(D)$ (the set of functions in $H^1(D)$ with trace zero).\\

We also remark that
$T\mathfrak H_\diamond \subset \mathfrak H_\diamond $
and hence the  restriction of $T$ to  $\mathfrak H_\diamond $
defines a linear bounded operator.  Since we are interested in
harmonic functions  in  $D$ and $D^\prime$ we only consider
the action of $T$ on  the closed space $\mathfrak H_\diamond $. We 
further keep the notation $T$ for the restriction of $T$ to 
$\mathfrak H_\diamond $.  We will prove later that $T$ has
only isolated eigenvalues with an accumulation point  $1/2$. 
We denote by $\left(\lambda_n^-\right)_{n\geq 1}$ the eigenvalues of
$T$   repeated according to their multiplicity, and 
 ordered as follows
\bea
 0<\lambda_1^- \leq \lambda_2^- \leq \cdots < \frac{1}{2},
\eea
 in $(0, 1/2]$ and, similarly,
 \bea
1> \lambda_1^+ \geq \lambda_2^+ \geq \cdots > \frac{1}{2}.
\eea
 the eigenvalues in $[1/2, 1)$. The eigenvalue $1/2$ is the unique accumulation
point of the spectrum. 
 \begin{remark}
 In contrast with the Dirichlet Poincar\'e variational spectral problem,  $0$  
 is not an eigenvalue of $T$. In fact if $w$ is an eigenfunction
 associated to zero, then it satisfies 
  \bea
\left \{
\ba{lllcc}
\Delta w(x) =  0 & \textrm{in}
\quad D^\prime,\\
\nabla w(x) = 0& \textrm{in}
\quad D,\\
 \partial_{\nu_\Omega} w(x)=  0  & \textrm{on}
\quad \partial \Omega,\\
\int_{\partial \Omega} w(x) ds(x) = 0. &  
\ea
\right.
\eea
Since this system of equations has only the trivial solution, zero is not
in the point spectrum of $T$.
\end{remark}
Next, we will characterize the spectrum of  $T$ via
the mini-max principle.  
\begin{proposition}  The variational Poincar\'e operator has the
following decomposition
\bean \label{Tdecomp}
T= \frac{1}{2}I +K,
\eean
where $K$ is a compact self-adjoint operator.  Let  $w_n^\pm, \; n\geq 1$ 
be the  eigenfunctions associated to the eigenvalues 
$\left(\lambda_n^-\right)_{n\geq 0}$. Then 
\bea
\lambda_1^-&=& \min_{0\not= w\in  \mathfrak H_\diamond}
\frac{\int_D|\nabla w(x)|^2 dx}{\int_\Omega |\nabla w(x)|^2 dx},\\
\lambda_n^-&=& \min_{0\not=w \in  \mathfrak H_\diamond, w\perp w_1,\cdots, w_{n-1} }
\frac{\int_D |\nabla w(x)|^2 dx}{\int_\Omega |\nabla w(x)|^2 dx},\\ 
&=& \min_{F_n\subset  \mathfrak H_\diamond, \; dim(F_n) = n}\max_{w\in F_n}
\frac{\int_D |\nabla w(x)|^2 dx}{\int_\Omega |\nabla w(x)|^2 dx},\\
\eea

and similarly 
\bea
\lambda_1^+&=& \max_{0\not=w\in  \mathfrak H_\diamond}
\frac{\int_D|\nabla w(x)|^2 dx}{\int_\Omega |\nabla w(x)|^2 dx},\\
\lambda_n^+&=& \min_{0\not=w \in  \mathfrak H_\diamond, w\perp w_1,\cdots, w_{n-1} }
\frac{\int_D |\nabla w(x)|^2 dx}{\int_\Omega |\nabla w(x)|^2 dx},\\ 
&=& \max_{F_n\subset  \mathfrak H_\diamond, \; dim(F_n) = n}\min_{w\in F_n}
\frac{\int_D |\nabla w(x)|^2 dx}{\int_\Omega |\nabla w(x)|^2 dx}.\\
\eea
\end{proposition}
\begin{proof}
We follow the approach of \cite{BT} for the spectrum of the Poincar\'e operator 
in the whole space.\\

Define the operator $K:   \mathfrak H_\diamond \to   \mathfrak H_\diamond$ by
\bean \label{operatorK}
2\int_\Omega \nabla Ku \cdot \nabla v dx = \int_D \nabla u \cdot \nabla v dx - 
\int_{D^\prime} \nabla u \cdot \nabla v dx. 
\eean
Then $K$ is a bounded self-adjoint operator  with norm $\|K\|\leq 1$. 
The first step of the proof is to show that $K$ is indeed a compact 
operator. \\

Let $\mathcal N(x,z)$ be the Neumann function for the Laplacian  in
$\Omega$, that is, the solution to
 \bean \label{systemN}
\left \{
\ba{lllcc}
\Delta \mathcal N(x,z) =  \delta_z & \textrm{in}
\quad \Omega,\\
 \partial_{\nu_\Omega} \mathcal N(x, z)= \frac{1}{|\partial \Omega|}   
& \textrm{on}
\quad \partial \Omega,\\
\int_{\partial \Omega} \mathcal N(x, z) ds(x) = 0, &
\ea
\right.
\eean
 where $\delta_z$ is the Dirac mass at $z$. 
 
 Define the single layer potential  
 $S_D: H^{-\frac{1}{2}}(\partial D)\rightarrow  \mathfrak H_\diamond$ by
 \bea
 S_D[\varphi](x) = \int_{\partial D} \mathcal N(x, z)  \varphi(z) ds(z).
 \eea
 Since the Neumann function   and the Laplace Green function in the
 whole space have equivalent weak singularities as $x\to z$, 
 (see for instance Lemma 2.14 in \cite{AK}) the operator
 $S_D$ satisfies the same jump relations through 
 the boundary of $D$ as the single layer of the Laplace Green function,
 that is,  
 \bea
 \partial_{\nu_D} S_D[\varphi](x)|^\pm = \pm \frac{1}{2}\varphi(x) +K^*_D[\varphi](x),  
 \eea
  for $x\in \partial D$, where $K^*_D:
   H^{-\frac{1}{2}}(\partial D)\rightarrow H^{-\frac{1}{2}}(\partial D)$, defined by
 \bea
 K^*_D[\varphi](x) = \int_{\partial D} \partial_{\nu_D(x)} \mathcal 
N(x, z)  \varphi(z) ds(z), 
 \eea
 is a compact operator. Here, $H^s(\partial D)$ are the usual Sobolev spaces on $\partial D$.  \\
 
 It can also be shown that  $S_D: 
 H^{-\frac{1}{2}}(\partial D)\rightarrow H^{\frac{1}{2}}(\partial D)$ 
 is invertible (this result is not true in general for the single layer
  the Laplace Green function in dimension two. Nevertheless, the operator $S_D$ can 
  be slightly modified to ensure invertibility \cite{ando}). \\
 
 Integrating by parts over $D$ and $D^\prime$ in \eqref{operatorK}, 
 using the jump conditions  and the fact that 
 $u$ lies in $\mathfrak H_\diamond$, we obtain 
 \bea
\int_\Omega \nabla Ku \cdot \nabla v dx = 
\int_{\partial D} K^*_D\left[S^{-1}_D[u|_{\partial D}]\right]v 
ds(x).\eea
Since $K^*_D$ is compact,  the operator  $K$
is also compact. \\

A direct calculation shows that the operator $T$ has the following decomposition
\bea
T= \frac{1}{2}I +K.
\eea
Then $T$ is Fredholm operator of index zero and enjoy the same spectral
decomposition as well as the min-max principle than the self-adjoint
and compact operator $K$.

\end{proof}

\begin{remark}
We first  remark that this result does not hold 
true if $D$ is merely Lipschitz. Finally, the space $\mathfrak H_\diamond$ 
can be defined as follows
\bea
\mathfrak H_\diamond := \left\{ S_D[\varphi]; 
 \;\; \varphi \in H^{-\frac{1}{2}}(\partial D) \right\}.
\eea
Considering this characterization, it is clear that  $\mathfrak H_\diamond$  
is a closed subspace in $H^1_\diamond(\Omega)$.
\end{remark}
We further normalize the eigenfunctions   $w_n^\pm, \; n\geq 1$ in
$\mathfrak H_\diamond$.  A direct consequence of the previous
result is the following  spectral decomposition  of functions
in  $\mathfrak H_\diamond$.

\begin{corollary}
Let $u$ be in $\mathfrak H_\diamond$. Then $u$ has the following 
spectral decomposition in $\mathfrak H_\diamond$:

\bea
u(x) = \sum_{n=1}^\infty  u_n^\pm w_n^\pm(x),
\eea  
where 
\bea
u_n^\pm = \int_\Omega \nabla u(x) \cdot \nabla w_n^\pm(x) dx.
\eea
\end{corollary}

A similar spectral  decomposition also holds for the
Neumann function.

\begin{corollary}
Let $\mathcal N(x,z)$ be the  Neumann function defined 
in~\eqref{systemN}. Then 

\bea
\mathcal N(x,z) =  - \sum_{n=1}^\infty  w_n^\pm(x) w_n^\pm(z),
\eea  
for  all $x, z \in \Omega$ such that  $x\not= z$.
\end{corollary}

\section{Frequency dependence of the boundary data}

We will first study the regularity of the solution
 $u(x, \omega)|_{\partial \Omega}$ as
a function of the frequency function $k(\omega)$. We show that
it is indeed meromorphic with poles of finite order. Then, we use the 
unique continuation property of  meromorphic complex functions
to determine the position of the poles and their corresponding 
singular parts. \\
It turns  out that the poles are related to the plasmonic resonances
of the inclusion \cite{pierre,ando, ando2}. We finally retrieve the
non-frequency part of the potential 
from the plasmonic  spectral information.    \\

\subsection{Spectral decomposition of the solution $u(x, \omega)$}
We have the following decomposition of $u(x, \omega)$ in the basis
of the eigenfunctions of the variational Poincar\'e operator $T$.  
\begin{theorem} \label{freqdepend}
Let $u(x, \omega)$ be the unique solution to the system \eqref{maineq}. \\

Then the following decomposition holds:
\bean
u(x, \omega) = k_0^{-1} u_0(x) +   \sum_{n=1}^\infty 
\frac{  \int_{\partial \Omega} f(z) w_n^\pm(z) ds(z) 
}{k_0+\lambda_n^\pm(k(\omega) -k_0) }  w_n^\pm(x),\quad x\in \Omega,
\eean
where $u_0(x) \in H^1_\diamond (\Omega)$ 
depends only on $f$ and $D$, and is the unique solution to 
\bean \label{nonfrequencypart}
\left \{
\ba{lllcc}
\Delta v = 0 & \textrm{in}
\quad D^\prime,\\
\nabla v = 0 & \textrm{in}
\quad D,\\
 \partial_{\nu_\Omega} v = f & \textrm{on}
\quad \partial \Omega.
\ea
\right.
\eean

\end{theorem}
\proof

Let $\f$ be the unique solution in $H^1_\diamond(\Omega)$ 
 to 
\bean \label{frakfunction}
\left \{
\ba{lllcc}
\Delta \f = 0 & \textrm{in}
\quad \Omega,\\
 \partial_{\nu_\Omega}\f = f & \textrm{on}
\quad \partial \Omega.
\ea
\right.
\eean 
The function $\f$ can be written in terms of the
Neumann function as follows:   
\bea
\f(x) = \int_{\partial \Omega} \mathcal N(x, z) f(z) ds(z). 
 \eea
 Denote $\mathfrak u :=  u- k_0^{-1}\mathfrak f $.  Then $\u$ lies
in $\H_\diamond$, and has the following spectral decomposition:

\bea
\u(x) = \sum_{n=1}^\infty  \u_n^\pm w_n^\pm(x),
\eea  
where 
\bea
\u_n^\pm = \int_\Omega \nabla \u(x) \cdot \nabla w_n^\pm(x) dx.
\eea

On the other hand,  $\u(x, \omega)$ is the unique solution to 
\bean \label{systemlast}
\left \{
\ba{lllcc}
-\nabla \cdot \left(\sigma(x, \omega) \nabla \u(x, \omega)\right) = k_0^{-1}
\nabla \cdot \left(\sigma(x, \omega) \nabla \f\right)    & \textrm{in}
\quad \Omega,\\
\sigma(x, \omega) \partial_{\nu_\Omega}\u(x, \omega)= 0 & \textrm{on}
\quad \partial \Omega,\\
\int_{\partial \Omega} \u(x, \omega) ds = 0.&
\ea
\right.
\eean 
Multiplying the first equation in (\ref{systemlast}) by $w_n^\pm(x)$,
and integrating by parts over $\Omega$, we get 
\bea
\u_n^\pm = \frac{k_0^{-1}
\int_\Omega  \nabla \cdot \left(\sigma(x, \omega) \nabla \f\right) w_n^\pm dx  
 }{k_0+\lambda_n(k(w) -k_0) }.
\eea
Since $ \nabla \cdot \left(\sigma(x, \omega) \nabla \f\right) $ lies 
in $H^{-1}(\Omega)$, the integral
 in the fraction above can be understood as a dual product 
 between $H^{-1}(\Omega)$
 and $H^{1}(\Omega)$, and can be simplified through 
 integration by parts into
\bea
\int_\Omega  \nabla \cdot \left(\sigma(x, \omega) \nabla \f\right) w_n^\pm dx  
= - (k_0+\lambda_n(k(\omega) -k_0))\frac{1}{\lambda_n} 
\int_D \nabla \f \cdot \nabla w_n^\pm dx +
k_0\int_{\partial \Omega} f w_n^\pm ds(x).
\eea
Consequently, it follows that 
\bea
\u_n^\pm =  -\frac{\int_D \nabla \f \cdot \nabla w_n^\pm dx}{k_0\lambda_n^\pm} 
+ \frac{ \int_{\partial \Omega} f w_n^\pm ds(x)
 }{k_0+\lambda_n^\pm(k(\omega) -k_0) }.
\eea
Now we derive the orthogonal projection of $\f$ onto $\H_\diamond$.
Let $\tilde f(x) \in \H_\diamond$ be 
the function that coincides with $\f$ on $D$ up to a constant, and solves  
the system of equations
\bea
\left \{
\ba{lllcc}
\Delta \tilde f = 0 & \textrm{in}
\quad D^\prime,\\
\\
 \nabla \tilde f = \nabla \f  & \textrm{in}
\quad  D,\\
 \partial_{\nu_\Omega}\tilde f = 0 & \textrm{on}
\quad \partial \Omega.
\ea
\right.
\eea
Since $w_n^\pm$ is an eigenfunction of $T$ and 
$\tilde f$ belongs to $\H_\diamond$, we have 
\bea
\int_D  \nabla \f \cdot \nabla w_n^\pm dx = 
 \lambda_n^\pm \int_\Omega \nabla \tilde f\cdot  \nabla w_n^\pm dx,
\eea
which gives 
\bea
\u_n^\pm =  -k_0^{-1} \int_\Omega \nabla \tilde f \cdot \nabla w_n^\pm dx+ 
\frac{  \int_{\partial \Omega} f w_n^\pm ds(x)
 }{k_0+\lambda_n(k(\omega) -k_0) }.
\eea
Finally, we obtain the desired decomposition for $u(x, \omega)$.

\endproof 

\begin{corollary} \label{holom}
The function $ u(x, \omega) =k_0^{-1} u_0(x) + u_f(x, k(\omega))$, where 
 $k\to u_f(x, k)$ is meromorphic on 
$\mathbb C$.  Furthermore, the
poles  of  $u_f(x, k)$ are the complex values 
$(k_n^\pm)_{n \geq 1}$ solutions to the dispersion equations
\bea
k_0+ \lambda_n^\pm(k -k_0) = 0, \qquad n\geq 1
\eea
with $\lambda_n^\pm, n\geq 1$ being the eigenvalues of the variational
 Poincar\'e
operator $T$.
\end{corollary}

\subsection{Retrieval of  the frequency dependent part.} 
The idea here is to recover  the
 frequency dependent part $u_f(x, k(\omega))$
from the knowledge of $u(x, \omega)$   for $\omega \in (\underline \omega,
\overline \omega)$. \\ 

The poles  of $u_f(x, k)$ are  given by 
$  k_n^\pm := k_0(1-\frac{1}{\lambda_n^\pm})$, and they can be  
 ordered as follows:
\bea
 -k_0<\cdots \leq k_2^+ \leq k_1^+ < 0
\eea
 in $(-k_0, 0)$ and, similarly,
 \bea
 k_1^- \leq k_2^- \leq \cdots <-k_0
\eea
 in $(-\infty, -k_0)$. We remark that $-k_0$  is the only accumulation
 point of the sequence of poles, i.e.,  $k_n^\pm$  tends to $-k_0$
 as $n\to \infty$. \\
 
The  plasmonic resonances  $\left(k_n^{\pm}\right)_{n\geq 1} $ 
 depend only on $k_0$, the shapes of the inclusion $D$ and the background $\Omega$. 
 They can be 
 experimentally measured and represents the plasmonic signature of 
 the inclusion.  One interesting inverse problem is to recover the 
inclusion from 
 its plasmonic resonances \cite{zou}. The magnitude of
$k_1^{-}$ is related somehow to how flat is 
the domain $D$. More precisely, we have the following result. 

\begin{lemma} \label{eigenvalueslowerbound}
There exists a constant $ \hat \delta>0$
depending only  on $k_0$  and $\mathfrak D$   such 
that 
\bea
k_n^\pm\geq -\hat \delta^{-1}, \qquad \forall  n\geq 1.
\eea

\end{lemma}

\proof

Let $D$ be an inclusion  in $\mathfrak D$. Then  $D$ is star-shaped and
is given by 
$$ D :=\{x\in \mathbb R^d : |x|< 
\Upsilon(\hat x), \; \hat x = \frac{x}{|x|} \}, $$
where  $\Upsilon: \mathbb S^{d}\rightarrow \mathbb (b_0, b_1 -\delta)$,
is $C^{2,\varsigma}$, $\varsigma >0$. \\

A  forward computation shows that  the constant $$r_D := \inf_{x\in 
\partial D} x\cdot \nu_D(x),$$
is strictly positive, and is lower and upper  bounded by constants  
that depend only 
on $\mathfrak D$.\\

On the other hand, a simple modification of the proof of Theorem 2.2 in \cite{AS}
(see also Lemma 2.9 in  \cite{AK}),
gives 
 
\bea
- \infty <-1- \left(\frac{r_D+2}{r_D}\right)^2  \leq k_1^{-},
\eea
which completes the proof. 


\endproof

Since the function $u_f(x, k)$ have isolated poles,  the complementary
of the singular set is connected and  the unique continuation of 
holomorphic 
functions implies the uniqueness  in the identification of 
the poles $k_n^\pm, n\geq 1$,   $k_0^{-1} u_0(x) $ and 
$u_f(x, k(\omega))$. \\

In order to derive stability estimates in the retrieval of the 
frequency independent part $ k_0^{-1} u_0(x)$ of the solution, we 
need to  obtain uniform bounds  on the frequency dependent part $u_f$
on the boundary $\partial \Omega$.

\begin{theorem} \label{thmbounds}
Let $D$ be an inclusion in $\mathfrak D$.
 Then there exists a constant 
$C= C(\mathfrak D,  \Omega, k_0)>0$ such that 
\bean \label{ukbound}
\|u(x,\omega)-k_0^{-1} u_0(x)\|_{C^0(\partial \Omega)} 
\leq \frac{C}{\textrm{dist}(k(\omega), [k_1^{-}, 0])}
 \|f(x)\|_{{H^{ -\frac{1}{2}}(\partial \Omega) }}.
\eean
The constant $C$ tends to $+\infty$ as $\hat \delta$ tends to zero. 
\end{theorem}
\proof 
Recall that the  function  $u_f(x, k) = u(x,\omega)- k_0^{-1} 
u_0(x)$,  defined in Corollary \ref{freqdepend}, lies in 
$\mathfrak H_\diamond$, and satisfies 
\bea
k_0\partial_{\nu_D} u_f|^+_{\partial D} - k\partial_{\nu_D} u_f|^-_{\partial D} = 
-\partial_{\nu_D} u_0|^+_{\partial D}.
\eea
   Hence there exists a potential
$\varphi_f \in H^{-\frac{1}{2}}(\partial D)$
satisfying 
\bea
u_f(x,k) = S_D[\varphi_f](x),
\eea
for $x\in \Omega$. Note that  the right-hand side term in the equality
above is harmonic in
both $D$ and $D^\prime$ and continuous through the boundary $\partial D$.
 The transmission  condition of $\partial_{\nu_D} u_f|_{\partial D}$ 
 over $\partial D$ implies
\bea
(\frac{k_0+k}{2(k_0-k)} I + K^*_D)[\varphi_f](x) &=&
 \frac{1}{k-k_0} \partial_{\nu_D} u_0(x)|^+_{\partial D},
\eea
for $x\in \partial D$.\\

On the other hand,  Calderon's identity  holds for the
operators $K_D, K^*_D$ and $S_D$, and we have 
\bea
S_D K^*_D &=& K_D S_D 
\eea
Hence,  $K^*_D$  becomes  a self-adjoint compact operator
 in the topology induced by the  scalar product $$
 \langle 
 \cdot \, ,  \cdot
  \rangle_{-\frac{1}{2}, S} = \langle -S_D 
 \cdot \, ,  \cdot
  \rangle_{\frac{1}{2}, -\frac{1}{2}}.$$    
A direct calculation shows that $\|K_D^*\|= 1$ and the spectrum of $K_D^*$
lies in $(-\frac{1}{2}, \frac{1}{2}] $. \\

Moreover,
the eigenvalues of $K^*_D$  are given by
\bea
0, \;\;\frac{1}{2},\;\; \frac{k_0+k_n^\pm}{2(k_0-k_n^\pm)},  \; n\geq 1. 
\eea
Spectral decomposition of self-adjoint compact operator shows that
\bean \label{sinequality}
\|\varphi_f\|_{-\frac{1}{2}, S} \leq  \frac{1}{\textrm{dist}(k, [k_1^{-}, 0])}  
\frac{2}{k_0} \|  \partial_{\nu_D} u_0|_+ \|_{-\frac{1}{2}, S}.
\eean

In order to derive precise estimates with constants that  depend only on
$\hat \delta$ and $\Omega$, we introduce the more conventional
 $H^{\frac{1}{2}}$-norm:
\bea
\|\psi\|_{\frac{1}{2}} =  \|v_\psi\|_{H^1(D)},
\eea
where $v_\psi$ is harmonic in $D$, that is, $\Delta v_\psi = 0$ on $D$,  and 
satisfies
$v_\psi|_{\partial D} = \psi$ on $\partial D$. \\

Following \cite{McL}, we define the associated $H^{-\frac{1}{2}}$-norm by
\bea
\| \varphi \|_{-\frac{1}{2}}  = \max_{ 0\not= \psi \in H^{\frac{1}{2}}(\partial D)}
  \frac{ \left| \int_{\partial D}  \varphi  \psi ds \right| }{ \| \psi\|_{\frac{1}{2}} }.
\eea
Now, we shall estimate $\|\varphi_f\|_{-\frac{1}{2}}$ in terms of the
quantity $\|\varphi_f\|_{-\frac{1}{2}, S}$. We have
\bea
\| \varphi_f \|_{-\frac{1}{2}}  = \max_{ 0\not= \psi \in H^{\frac{1}{2}}(\partial D)}
  \frac{ \left| \int_{\partial D}  \varphi_f  \psi ds \right| }{ \| \psi\|_{\frac{1}{2}} }\\
  \leq  \max_{ 0\not= \psi \in H^{\frac{1}{2}}(\partial D)}
  \frac{ \left| \int_{\Omega}  \nabla S_D[\varphi_f] \cdot \nabla \tilde v_\psi  dx \right| }{
   \| \tilde v_\psi \|_{H^1(D)} },
\eea

where $\tilde v_\psi \in H^1_\diamond (\Omega)$ 
is the unique solution to 
\bea
\left \{
\ba{lllcc}
\Delta v = 0 & \textrm{in}
\quad D^\prime,\\
\Delta v = 0 & \textrm{in}
\quad D,\\
v= \psi & \textrm{on}
\quad \partial D, \\
 \partial_{\nu_\Omega} v = f & \textrm{on}
\quad \partial \Omega.
\ea
\right.
\eea
Hence,
\bea
\| \varphi_f \|_{-\frac{1}{2}}    \leq   \left(\int_{\Omega} \left| \nabla 
S_D[\varphi_f]\right|^2 dx\right)^{\frac{1}{2}}
\left(\max_{ 0\not= \psi \in H^{\frac{1}{2}}(\partial D)}
  \frac{  \int_{\Omega} | \nabla \tilde v_\psi |^2 dx  }{
   \int_{D} | \nabla \tilde v_\psi |^2 dx }\right)^{\frac{1}{2}}\\
   \leq \left(\lambda_1^-\right)^{-\frac{1}{2}} \| \varphi_f \|_{-\frac{1}{2}, S}\\ \leq 
   \left(1- \frac{k_1^-}{k_0} \right)^{\frac{1}{2}}  \| \varphi_f \|_{-\frac{1}{2}, S}.
   \eea
   Using the inequality satisfied by  $k_n^\pm$  
   in Lemma \ref{eigenvalueslowerbound}, we obtain 
   \bea
  \| \varphi_f \|_{-\frac{1}{2}}    \leq  
   C\left(1+\frac{1}{\hat \delta k_0} \right)^{\frac{1}{2}}  \| \varphi_f \|_{-\frac{1}{2}, S},
\eea
where $C$ depends only on $\Omega$.  
Combining the last inequality  and (\ref{sinequality}), we get
\bean\label{s2inequality}
\| \varphi_f \|_{-\frac{1}{2}}    \leq C \left(1+\frac{1}{\hat \delta k_0} \right)^{\frac{1}{2}}  
\frac{1}{\textrm{dist}(k, [k_1^{-}, 0])}  
\frac{1}{k_0} \|  \partial_{\nu_D} u_0|_+ \|_{-\frac{1}{2}, S}.
\eean
Next, we estimate  $\|  \partial_{\nu_D} u_0|_+ \|_{-\frac{1}{2}, S}$
in terms of $\|f\|_{H^{ -\frac{1}{2}}(\partial \Omega) }$. \\

A direct calculation shows that  $$u_0- S_D[ \partial_{\nu_D} u_0|_+] = \frak f,$$ 
over $\Omega$. \\

Therefore,
\bean \label{zero}
\|  \partial_{\nu_D} u_0|_+ \|_{-\frac{1}{2}, S} = \int_{\Omega} \left| \nabla
S_D[ \partial_{\nu_D} u_0|_+]  \right|^2 dx
\leq  \int_{\Omega} \left| \nabla
u_0 \right|^2 dx +  \int_{\Omega} \left| \nabla
\frak f \right|^2 dx. 
\eean

On the other hand, we have
\bea
 \int_{\Omega} \left| \nabla
u_0 \right|^2 dx = - \int_{\partial \Omega } f u_0 ds \leq C_1
 \|f\|_{H^{ -\frac{1}{2}}(\partial \Omega) }
\left(\int_{\Omega} \left| \nabla
u_0 \right|^2 dx \right)^{\frac{1}{2}},
\eea
where $C_1>0$ is the constant that appears in the trace theorem 
on $\partial \Omega$ and depends only on $\Omega$ and the 
dimension of the space.  
Hence, 
\bean \label{first}
 \int_{\Omega} \left| \nabla
u_0 \right|^2 dx 
\leq C_1^2 \|f\|_{-\frac{1}{2}}^2.
\eean
Since $\frak f $ is the unique solution to the 
system \eqref{frakfunction}, classical 
elliptic regularity implies
\bean \label{second}
 \int_{\Omega} \left| \nabla
\frak f \right|^2 dx \leq  C_2
 \|f\|_{-\frac{1}{2}}^2,
\eean
where  $C_2>0$ is a constant which depends only on 
$\Omega$ and the 
dimension of the space.  \\
Combining inequalities \eqref{zero}, \eqref{first}, and 
\eqref{second}, we obtain
\bean \label{final}
\|  \partial_{\nu_D} u_0|_+ \|_{-\frac{1}{2}, S} \leq 
C  \|f\|_{-\frac{1}{2} }. 
\eean

Now we turn to inequality \eqref{s2inequality}. Using the estimate 
above, we get

\bean\label{s3inequality}
\| \varphi_f \|_{-\frac{1}{2}}    \leq C
 \left(1+\frac{1}{\hat \delta k_0} \right)^{\frac{1}{2}}  
\frac{1}{\textrm{dist}(k, [k_1^{-}, 0])}  
\frac{1}{k_0} \|f\|_{H^{ -\frac{1}{2}}(\partial \Omega) },
\eean
where $C$ depends only on 
$\Omega$.  \\

Now, we are ready to prove the results of the theorem. Using the 
fact that  
$\textrm{dist} \left( D, \partial \Omega \right) > \delta $, and
\bea
u_f(x) =  \int_{\partial D}  \mathcal N(x, z) 
\varphi_f(z) ds(z),  
 \eea
for all $x\in \partial D$, we deduce that 

\bea
\|u_f(x)\|_{C^0(\partial \Omega)} \leq  \left( \max_{x\in \partial \Omega} 
\| \mathcal N(x, .) \|_{H^1(\Omega_\delta) }\right)
\| \varphi_f \|_{-\frac{1}{2}},
\eea
where $\Omega_\delta=  \{x\in \mathbb R^d:  \textrm{dist}(x, \partial \Omega) 
>\delta\}$.\\

Finally, by using estimate \eqref{s3inequality} we achieve the proof of the 
theorem.

\endproof
\begin{remark}
The results of Theorem \ref{thmbounds} give a precise estimate on how the
solution $u$ to \eqref{maineq} blows up on the boundary  when 
$k(\omega)$ approaches the plasmonic resonances. The estimates are
uniform  for inclusions within the set $\mathfrak D$, and are somehow a 
generalization of the results in \cite{KKL} which are only valid in a
 sector of the complex plane. 
\end{remark}

\begin{theorem}  \label{thmbounds2}
Let $D$ be an inclusion in $\mathfrak D$.
 Then there exists a constant 
$C= C(\mathfrak D, k_0, \Omega)>0$ such that 
\bean \label{ukbound2}
\|u_0(x)\|_{C^0(\partial \Omega)} 
\leq C
\left(1+\frac{1}{\textrm{dist}(\Sigma, [k_1^{-}, 0])}\right)
 \|f(x)\|_{{H^{\frac{1}{2}}(\partial \Omega) }}.
\eean
The constant $C$ tends to $+\infty$ as $\hat \delta$ tends to zero. 

\end{theorem}
\proof Recall that $$u_0- S_D[ \partial_{\nu_D} u_0|_+] = \frak f.$$ 
Then, similarly to the proof of  
Theorem \ref{thmbounds},  based on elliptic
regularity, and the uniform bound \eqref{final} of 
$ \partial_{\nu_D} u_0|_+$, we deduce
the desired result.

\endproof

Next, we show   that the stability of the reconstruction of  $u_f$ 
 depends  in fact on the distance  
 of the poles  $k_n^\pm $ to the set
 $\Sigma= \{k(\omega);  \omega \in (\underline \omega,
\overline \omega) \}$ in the complex plane.\\
\begin{theorem}  \label{estimationuf} Let $D$ and $\widetilde D$ 
be two  inclusions 
in $\mathfrak D$. 
Denote by $u$ (resp. $\tilde u$), the solution of \eqref{maineq} with 
inclusion
$D$ (resp. $\widetilde D$). Let 
$$\varepsilon = \sup_{x \in \partial \Omega, \omega \in (\underline \omega,
\overline \omega)}|u-\tilde u |.$$ 
Then,  there exists a constant $\kappa>0$, that depends only on 
$ \Omega, \mathfrak D, k_0$, and $\Sigma$, 
such that 
\bean \label{estimateuf}
\sup_{x \in \partial \Omega, \omega \in (\underline \omega,
\overline \omega)}|u_f-\tilde u_f| \leq C\varepsilon^\kappa,
\eean
where the constant $C>0$ depends only on $f, \Omega, \mathfrak D$, 
and $\Sigma$.
\end{theorem}
\proof

The proof of the theorem is based on the unique continuation 
of holomorphic
 functions. \\

For $x\in  \partial \Omega$ fixed,  Lemma~\ref{freqdepend} implies
that  $ 
\alpha(k) = k_0^{-1} u_0(x) + u_f(x, k)$, 
is a meromorphic function with poles $(k_n^\pm)_{n\geq 1}$. \\

Similarly for   $x\in  \partial \Omega$ fixed, we have 
  $ \tilde \alpha(k) =  k_0^{-1} \tilde u_0(x) + \tilde u_f(x, k)$, 
is a meromorphic function with poles
 $(\tilde k_n^\pm)_{n\geq 1}$ (the plasmonic surface 
resonances of the inclusion $\widetilde D$).\\ 

We consider $\tilde u(x,\omega)$ as a  perturbation of  the  
meromorphic function  $u(x,\omega)$ on $\mathbb C$. 
We will use the concept of harmonic measure 
to estimate 
 the difference $\alpha(k)- \tilde \alpha(k)$ on a complex
contour that encirles the poles of both functions $\alpha(k)$
and $ \tilde \alpha(k) $. \\

Let  $\mathcal C_+$ 
be a Jordan complex contour with interior  
$\overset{\circ}{\mathcal C_+}$ that contains $
[-\hat \delta^{-1}, 0)\cup\overline \Sigma $. Let  $\mathcal C_-$ be 
a Jordan complex contour in ${\overset{\circ}{\mathcal C_+}}$, with interior  
$\overset{\circ}{\mathcal C_-}$ that contains $[-\hat \delta^{-1}, 0]$ and
does not intersect $\Sigma$, that is, $\overset{\circ}{\mathcal C_-} 
\cap \Sigma =\emptyset$.
Finally, let $\mathcal C$ be a Jordan complex contour
in $ \overset{\circ}{\mathcal C_+}\setminus 
\overset{\circ}{\mathcal C_-}$
such that  $ [-\hat \delta^{-1}, 0] \subset \overset{\circ}{\mathcal C}$, and 
$\overset{\circ}{\mathcal C} \cap \Sigma =\emptyset$. \\

Let $\omega $ be a fixed frequency in 
$(\underline \omega, \overline \omega)$. Since  the poles
 $ (k_n^\pm)_{n\geq 1}, (\tilde k_n^\pm)_{n\geq 1}$ are inside 
$\overset{\circ}{\mathcal C}$, and  
$k(\omega)$ lies in the  exterior of $ \mathcal C$, 
we deduce from 
Lemma \ref{freqdepend}  that
\bea
 u_f(x, k(\omega)) - \tilde 
u_f(x, k(\omega)) =  \frac{1}{2i\pi}\int_{\mathcal C} \frac{\alpha(k)
  -\tilde \alpha(k)}{k-k(\omega)} dk.  
\eea

Consequently,

\bean \label{rouche}
\left|u_f(x, k(\omega)) - \tilde 
u_f(x, k(\omega))\right| \leq \frac{1}{\textrm{dist}(\Sigma, \mathcal C)}
\|\alpha(k)
  -\tilde \alpha(k)\|_{L^\infty(\mathcal C)}.
\eean  

Now,  define $w(z)$ to be the harmonic
measure
of $\overline \Sigma$ in $ \overset{\circ}{\mathcal C_+}\setminus 
 \overline{\overset{\circ}{\mathcal C_-}}$, which is holomorphic in 
$ \overset{\circ}{\mathcal C_+}\setminus  \overline
{\overset{\circ}{\mathcal C_-}} $ and 
statifies $w(z) =1 $ on $\overline \Sigma$, $ w(z) = 0$ on 
$ \mathcal C_- \cup \mathcal C_+$. 
\\

Then the two-constants theorem implies
\bea
|\alpha(k) -\tilde \alpha(k)| \leq M^{1-w(k)} \varepsilon^{w(k)},  
\eea
for all $z$ in $ \overset{\circ}{\mathcal C_+}\setminus 
\overset{\circ}{\mathcal C_-}$ where $M =  \max_{k\in 
  \overset{\circ}{\mathcal C_+}\setminus 
\overset{\circ}{\mathcal C_-}} (|\alpha(k)|+ |\tilde \alpha(k)|)
$. \\

We deduce from Theorems \ref{thmbounds} and  \ref{thmbounds2}  
that 
\bea
M \leq \frac{C}{\textrm{dist}([-\hat \delta^{-1}, 0], \mathcal C_-)} 
\|f\|_{H^{ \frac{1}{2}}(\partial \Omega) },
\eea
where $C>0$ depends only on $\mathfrak D, \Omega, \Sigma$  and $k_0$. \\

Taking $\kappa = \min_{k\in \mathcal C} w(k)$, we obtain 
\bean \label{f1}
\|\alpha(k)
  -\tilde \alpha(k)\|_{L^\infty(\mathcal C)} \leq C \varepsilon^{\kappa}
\eean
with $C>0$ being a constant that depends only on the contours 
$\mathcal C, \mathcal 
C_\pm, M $ and $\Sigma$.  \\

Combining the inequalities \eqref{rouche} and \eqref{f1}, we get the 
estimate \eqref{estimateuf} of the theorem.

\endproof

\begin{remark}
Since the position of $\Sigma$ in the complex plane is known,  
the contours $ \mathcal C_\pm$ and $\mathcal  C $ can be explicitly  given,
and the harmonic measure $w(z)$ can be explicitly  constructed using
known conformal  maps.  
Hence, the constant $\kappa$ can be precisely estimated in terms
of the distance between  the sets $\Sigma$ and $[-\hat \delta^{-1}, 0]$. 

\end{remark}
A direct consequence  of  the  Theorem \ref{estimationuf}  is the 
estimation of the frequency independent part of the data from the complete
collected data over $\Sigma$. 

\begin{corollary} \label{estimationu0}
Let $D$ and $\widetilde D$ be two  inclusions 
in $\mathfrak D$. 
Denote $u$ (resp. $\tilde u$), the solution of \eqref{maineq} with 
inclusion
$D$ (resp. $\widetilde D$). Let 
$$\varepsilon = \sup_{x \in \partial \Omega, \omega \in (\underline \omega,
\overline \omega)}|u-\tilde u |.$$ 
Then,  there exists a constant $\kappa>0$, that depends only on 
$ \Omega, \mathfrak D$ and $\Sigma$, 
such that 
\bean \label{estimateu0}
\sup_{x \in \partial \Omega, \omega \in (\underline \omega,
\overline \omega)}|u_0-\tilde u_0| \leq C\varepsilon^\kappa,
\eean
where the constant $C>0$ only depends on 
$f, \Omega, \mathfrak D$ 
and $\Sigma$.

\end{corollary}

\section{Reconstruction of the inclusion from the Cauchy
data of the  perfectly conductor
solution.}
In this section we construct the inclusion $D$
from the knowledge of the Cauchy data of
the frequency independent part
$u_0$ on the boundary $\partial \Omega$. Precisely, we
 derive 
the uniqueness and  stability  of the reconstruction
within the set of inclusions $\mathfrak D$. The
results are quite surprising in inverse conductivity 
problem since in general  infinitely many { input 
currents }
are needed
in order to  obtain the  uniqueness in the determination
of the conductivity. Here the
fact that the solution is constant inside the inclusion
is essential to derive such results. \\

Based on quantitative estimates of the unique continuation
for  Laplace operator  we evaluate how the solution on the 
boundary of the perturbed inclusion is sensitive to  errors 
made in the Cauchy data on $\partial \Omega$. Using the
fact that the 
solution is constant inside the inclusion
we then  obtain  the  variation of the solution on the
perturbed inclusion, and again using the interior
unique continuation \cite{GL, HL} we estimate the variation
of the inclusion induced by the errors in Cauchy data. The
methods developped here are similar to the ones used in 
determining parts of the boundaries under zero Dirichlet 
or Neumann conditions on unknown sub-boundaries 
\cite{BV, BCY1, 
ABRV, Rd1, Rd2, Is1}.   
 \\

The  conditional  stability  in  our  inverse  problem  depends
heavily  on  the one  in  a  Cauchy  problem  for  the  
Laplace  equation. 
 
\begin{lemma}  \label{firstCauchy}  
Let $D$ and $\widetilde D$ be two  inclusions 
in $\mathfrak D$. 
Let  $u_0$ (resp. $\tilde u_0$) be the solution in
$H^1_\diamond(\Omega)$ of 
\eqref{nonfrequencypart} with 
inclusion
$D$ (resp. $\widetilde D$),  and
assume
that

$$0 < \varepsilon = \sup_{x \in \partial \Omega}|u_0-\tilde u_0| < 1.$$

Then,  there exist constants $C>0$ and $\mu>0$,  
such that  the following
estimate holds:
\bean \label{stabilityestimate1c}
\left\|u_0 - \tilde u_0
\right\|_{C^0 \left(\partial\left(D\cup \widetilde D\right)
\right)} 
\leq C \left(\frac{1}{\ln(\varepsilon^{-1})}\right)^{\mu}.
\eean
Here, the constants $C$ and $\mu$ 
depend only on 
$f, \Omega,$ and $\mathfrak D$. \\

If, in addition $d=2$, and the inclusions  $D$ and $\widetilde D$
are analytic, then we have 
\bean \label{stabilityestimate12}
\left\|u_0 - \tilde u_0
\right\|_{C^0 \left(\partial\left(D\cup \widetilde D\right)
\right)} 
\leq C \varepsilon^{\mu^\prime},
\eean
where the constants $C$ and $\mu^\prime$ 
depend only on 
$f, \Omega,$ and  $\mathfrak D$. 

\end{lemma}
%

\proof
The stability estimate is well
known for smooth boundaries 
($\partial\left(D\cup \widetilde D\right)$
is Lipschitz).  The proof of this 
Lemma can be found in~\cite{BCY1, BCY2, 
CHY, Is1, Is2, Is3, BV} for the Laplacian
operator, and in \cite{ABRV} for an
elliptic operator in a divergence form. It
is based on three  facts. The first one is that
$\Omega\setminus \overline{D_1\cup D_2}$ satisfies
 a uniform cone property (see for instance
\cite{BCY1} for dimension two,  the proof
can be extended easily to higher dimensions). If 
$D_1$ and $D_2$ are  not star-shaped,  it is proved
in \cite{Rd1, ABRV} that 
$\Omega\setminus \overline{D_1\cup D_2}$ satisfies
the cone property if  $D_1$ and $D_2$ are too close
(which can be verified for general elliptic operators
 if $\varepsilon$ is too small
through a first rough stability  estimate using the
three sphere inequality \cite{GL}). The second fact is  
to evaluate the value of 
 $u_0 - \tilde u_0$ on a point $p\in 
\partial\left(D\cup \widetilde D\right)$ by evaluating the
unique continuation in a cone included in 
$\Omega\setminus \overline{D_1\cup D_2}$ with vertex $p$
using the explicit expression of the harmonic measure
(Lemma 4.2 in \cite{CHY}, Lemma 3.5 in  \cite{Rd1}, 
c) in proof of Lemma 3.6 in \cite{Is3}).  The third
fact is a H\"older Cauchy stability estimate 
in any smooth domain lying in
$\Omega\setminus \overline{D_1\cup D_2}$,  
neighboring $\partial \Omega$, and at a finite distance 
from the boundary  $\partial\left(D\cup \widetilde D\right)$
proved in \cite{Pa} (extended in \cite{Tr} 
for elliptic equations in a divergence form with  Lipschitz
coefficients).
Finally,  H\"older stability estimate type is obtained 
 for analytic curves in \cite{BCY1}.
 
\endproof

Recall that $\nabla u_0= 0$  (resp. $\nabla \tilde u_0 = 0$) in
$D$ (resp. $\tilde D$).  
We further denote 
$\varrho \;\;$ (resp. $\tilde \varrho$), 
the  constant value of   $u_0|_D \;\;$ (resp. $ \tilde u_0|_{\tilde D}$). 

\begin{lemma}  \label{secondCauchy}  
Let $D$ and $\widetilde D$ be two  inclusions 
in $\mathfrak D$. 
Let $u_0$ (resp. $\tilde u_0$) be  the solution in
$H^1_\diamond(\Omega)$ of 
\eqref{nonfrequencypart} with 
inclusion
$D$ (resp. $\widetilde D$). Then, the following 
estimate holds:
\bean \label{stabilityestimate2}
\left|\varrho - \tilde \varrho
\right| \leq  \left\| u_0 - \tilde u_0
\right\|_{C^0 \left(\partial\left(D\cup \widetilde D\right)
\right)}.
\eean

\end{lemma}
%
\proof
  Since $D$ and $\tilde D$ are
star-shaped and contain the point $0$, $D\cap \tilde
D$ is not empty.  Then we have two different cases.\\

Case (1): $\partial D\cap \partial \tilde D$ is not empty.
In this case the estimate is trivial. \\

Case (2): $\partial D\cap \partial \tilde D$ is empty, and hence
we have $D\subset \tilde D$ or $\tilde D\subset  D$. Without
any loss of generality, we will further assume that  $D\subset\tilde D$. 
By Green's formula  inside the domain $\Omega\setminus \overline
D$, we have $\int_{\partial D} \partial_{\nu_D} u_0(x)|^+ ds(x)  \,=\, 0$,
and hence $ \partial_{\nu_D} u_0(x)$ can not have a constant sign 
on $\partial D$. 
Since $u_0$ is constant on $\partial D$, 
we deduce then from  Hopf's Lemma that $u_0$ does not 
take its maximum or minimum on $\partial D$. The 
fact that $u_0$ is harmonic on  $\tilde D \setminus \overline
D$ implies that $u_0$ reaches its minimum and maximum 
in $\overline{\tilde D} \setminus D$  on $\partial \tilde D$. Consequently 
$u_0$ takes the value
$\varrho$ on $\partial \tilde D = \partial \left(D\cup \tilde D \right)$,
which concludes the proof of the estimate.

\endproof

Recall that  $u_0-\varrho$ (resp. $\tilde u_0-
\tilde \varrho$) is harmonic in $\Omega\setminus \overline D$ (resp.
 $\Omega\setminus \overline{ \tilde D}$,
and satisfies a zero Dirichlet boundary condition on $\partial D$
(resp. $\partial \tilde D$). Consequently,
\bea
\max_{x\in \tilde D\setminus \overline D}|u_0(x)-\varrho| \leq  \max_{x\in \partial 
\tilde D\setminus \overline D}
|u_0(x)-\tilde u_0(x)+\tilde \varrho-\varrho|,\\
\max_{x\in D\setminus \overline {\tilde D}}|\tilde u_0(x)-\varrho| 
 \leq  \max_{x\in \partial 
D\setminus \overline {\tilde D}} |\tilde u_0(x)-u_0(x)+ \varrho-\tilde \varrho|.
\eea
Using the estimate stated in Lemma \ref{secondCauchy}, we have
\bea
\max_{x\in \tilde D\setminus \overline D}|u_0(x)-\varrho| +
\max_{x\in D\setminus \overline {\tilde D}}|\tilde u_0(x)-\tilde \varrho| \leq
4\left\| u_0 - \tilde u_0
\right\|_{C^0 \left(\partial\left(D\cup \widetilde D\right)
\right)}.
\eea
Then we immediately obtain the following result. 

\begin{lemma}\label{thirdCauchy}
Let $D$ and $\widetilde D$ be two  inclusions 
in $\mathfrak D$. 
Let $u_0$ (resp. $\tilde u_0$)  be the solution in
$H^1_\diamond(\Omega)$ of 
\eqref{nonfrequencypart} with 
inclusion
$D$ (resp. $\widetilde D$). \\

Then, the following 
estimate holds:

\bean \label{stabilityestimate3}
\int_{D\setminus \overline{\tilde D}} 
|\tilde u_0(x)-\tilde \varrho|^2 dx + 
\int_{\tilde D\setminus \overline{ D}} |u_0(x)- \varrho|^2 dx 
\leq  C \left\| u_0 - \tilde u_0
\right\|_{C^0 \left(\partial\left(D\cup \widetilde D\right)
\right)}^2.
\eean
Here,  $C>0$ depends only on 
$\mathfrak  D$.

\end{lemma}

\begin{proposition} \label{doublinginequality}
Let $D$ and $\widetilde D$ be two  inclusions 
in $\mathfrak D$. 
Let $u_0$ (resp. $\tilde u_0$)  be the solution in
$H^1_\diamond(\Omega)$ of 
\eqref{nonfrequencypart} with 
inclusion
$D$ (resp. $\widetilde D$), and  $x_0 \in \partial D$ (resp.  
$x_0 \in \partial \tilde D$). \\

Then, for any $r>0$ and 
$R\geq r$,
we have 
\bea
\int_{B_{R}(x_0)\cap \Omega} |u_0(x)- \varrho|^2 dx 
 \leq C\left(\frac{R}{r}\right)^K
 \int_{B_{r}(x_0)\cap \Omega} |u_0(x)- \varrho|^2 dx, \\
 \int_{B_{R}( x_0)\cap \Omega} |\tilde u_0(x)- \tilde \varrho|^2 dx 
 \leq C\left(\frac{R}{r}\right)^K
 \int_{B_{r}(x_0)\cap \Omega} |\tilde u_0(x)- \tilde \varrho|^2 dx,
\eea
where  $C>1$ and $K>0$ depend on  $f, \Omega,\Sigma, $ and  $ \mathfrak D$. \\

Furthermore, we have  

\bea 
\int_{\Omega} |u_0(x)- \varrho|^2 dx \geq C_0,
\eea
where  $C_0>0$  depends on  $f, \Omega, \Sigma$, and  $ \mathfrak D$.
\end{proposition}
\proof 
The   doubling inequalities  are   obtained  in 
\cite{AdolfssonEscauriaza}
for general elliptic operators in a divergence form 
(Theorem 1.1). 
In \cite{ABRV}, a more explicit evaluation of the constants 
$C$ and $K$ in
 terms of the  problem a priori data is derived
 (Proposition
 4.5). 
\endproof

Now, we are ready to prove Theorems  
\ref{mainthm1} and   \ref{mainthm2}.
We follow the ideas developed  in  
the proof of Theorem 2.2  in \cite{ABRV}. 
\proof
For $\hat x \in \mathbb S^{d-1} $, where $\mathbb S^{d-1}$ is
 the unit sphere,  we further denote by
\bea
\Upsilon_m(\hat x) = \min(\Upsilon(\hat x), \tilde 
\Upsilon(\hat x)),\;\;\;
\Upsilon_M(\hat x) = \max(\Upsilon(\hat x), \tilde \Upsilon(\hat x)).
\eea
A direct computation shows that $\Upsilon_m$ and $\Upsilon_M$
belong to $C^{0,1}( \mathbb S^{d-1} )$. \\

We introduce the domains   
\bea
D_m(\hat x) = \left\{\ba{llcc}
D &\textrm{if}&  \Upsilon_m(\hat x) = \Upsilon(\hat x),\\
\tilde D &\textrm{if}&  \Upsilon_m(\hat x) = \tilde \Upsilon(\hat x),
\ea 
\right.
\eea
and 
\bea
D_M(\hat x) = \left\{\ba{llcc}
D &\textrm{if}&  \Upsilon_M(\hat x) = \Upsilon(\hat x),\\
\tilde D &\textrm{if}&  \Upsilon_M(\hat x) = \tilde \Upsilon(\hat x).
\ea
 \right.
\eea

 Let $r_M: \mathbb S^{d-1} \rightarrow \mathbb R_+$
be defined by

\bea
r_M(\hat x) = \max\{r:  0 \leq r 
\leq \Upsilon_M(\hat x)- \Upsilon_m(\hat x) \textrm{  and  } 
B_r(\Upsilon_m(\hat x)\hat x)\setminus \overline{D_m(\hat x) }
\subset  D_M(\hat x)\setminus \overline{D_m(\hat x) }  \}.
\eea

Then $r_M(\hat x)$ attains its maximum
$r_0>0$ over $\mathbb S^{d-1}$ at $\hat x_0$,  that is 
\bea
r_0 := r_M(\hat x_0) = \max_{\hat x \in \mathbb S^{d-1}} r_M(\hat x). 
\eea

Now, let $\hat x_M \in \mathbb S^{d-1}$,
be such that 
\bea
d_0 :=  \Upsilon_M(\hat x_M) - \Upsilon_m(\hat x_M)
= \max_{\hat x \in \mathbb S^{d-1}}  
\left(\Upsilon_M(\hat x) - \Upsilon_m(\hat x)\right). 
\eea

Obviously, we have $r_0\leq d_0\leq 2m$. 

\begin{lemma}  \label{differencediameter}
There exist constants $c_{1}>0$ and  $c_2>1$,  such that  the following 
estimates
hold:
\bea
 |D\Delta \tilde D| \leq  c_1m^{d-1} d_0,\\
 \min\left(b_0 \sqrt{\frac{d_0}{m}}, d_0\right) \leq c_2 r_0.
\eea
The constants $c_i, \; i=1,2, $ only depend  on  the dimension
of the space. 
\end{lemma}
\proof

Without any loss of generality, we assume that 
$\Upsilon_M(\hat x_M) =  \tilde \Upsilon(\hat x_M)$,
 $\Upsilon_m(\hat x_M) =  \Upsilon(\hat x_M)$, and 
  denote by $x_M =  \Upsilon(\hat x_M)\hat x_M
 \in \partial D$.\\

The first 
inequality immediately follows from the definition
 of $d_0$. \\
 
Now, from the $C^2$ regularity of the function
$\hat x \rightarrow \Upsilon(\hat x) - \tilde \Upsilon(\hat x)$, we 
deduce that for $t_0 = \frac{1}{2\sqrt{m}}$,
we have   $ d_0/2 \leq \Upsilon_M(\hat x) - \Upsilon_m(\hat x)\leq d_0$
for all $\hat x \in B_{t_0 \sqrt{d_0}}(\hat x_M) \cap \mathbb S^{d-1}$.\\

A forward calculation shows that for $ s_0= \frac{1}{4}\min\left(b_0 \sqrt{\frac{d_0}{m}}, d_0\right) $,
we have $$B_{s_0}(\Upsilon(\hat x_M)\hat x_M)\setminus 
\overline D \subset \tilde D \setminus 
\overline D.$$

 From  the definition of $r_0$, we  obtain

 $$ s_0
\leq r_0,$$  which finishes the proof of the lemma.

\endproof 

Without any loss of generality, we assume that 
$\Upsilon_M(\hat x_0) =  \tilde \Upsilon(\hat x_0)$,
 $\Upsilon_m(\hat x_0) =  \Upsilon(\hat x_0)$, and 
  denote by $x_0 =  \Upsilon(\hat x_0)\hat x_0
 \in \partial D$.\\

Now, let $\tilde r_0 >0$ such  that $B_{\tilde r_0}(x_0)\cap \Omega=
\Omega$.
 Then, we immediately have  $\tilde r_0> r_0$.  Obviously, $\tilde r_0$ 
only depends on $\mathfrak D$ and $\Omega$.\\

Taking $R= \tilde r_0$ and $r= r_0$ in Lemma \ref{thirdCauchy}, we 
obtain 

\bea
\int_{B_{\tilde r_0}(x_0)\cap \Omega} |u_0(x)- \varrho|^2 dx 
 \leq C\left(\frac{\tilde r_0}{r_0}\right)^K
 \int_{B_{r_0}(x_0)\cap \Omega} |u_0(x)- \varrho|^2 dx,
 \eea
which implies that
\bea
r_0^K \int_{\Omega} |u_0(x)- \varrho|^2 dx 
 \leq C \tilde r_0^K
 \int_{B_{r_0}(x_0)\cap \Omega} |u_0(x)- \varrho|^2 dx.
 \eea
Since $u_0(x)- \varrho$ vanishes inside $D$, and
$B_{r_0}(x_0) \setminus \overline D \subset \tilde D\setminus
  \overline D$,
we have 
\bea
r_0^K \int_{\Omega} |u_0(x)- \varrho|^2 dx 
 \leq C \tilde r_0^K
 \int_{B_{r_0}(x_0)\cap \Omega} |u_0(x)- \varrho|^2 dx,\\
  \leq C \tilde r_0^K
 \int_{\tilde D\setminus \overline D} |u_0(x)- \varrho|^2 dx. 
 \eea
From Lemma \ref{thirdCauchy}, we deduce that 
\bea
r_0  \left(\int_{\Omega} |u_0(x)- \varrho|^2 dx \right)^{\frac{1}{K}}
  \leq C  \left\| u_0 - \tilde u_0
\right\|_{C^0 \left(\partial\left(D\cup \widetilde D\right)
\right)}^{\frac{2}{K}}. 
 \eea
Now, combining estimates of Lemmas \ref{firstCauchy}, \ref{differencediameter},
 and  Proposition \ref{doublinginequality},
we finally obtain  the results of the main theorems.  

\endproof

\section{Acknowledgments}
This work has been partially supported by the 
LabEx PERSYVAL-Lab (ANR-11-LABX- 0025-01).


\end{document}